\documentclass[11pt]{amsart}
\usepackage{geometry} 
\usepackage{pgfplots}
\usepackage{psfrag}
\usepackage{amsmath}
\usepackage{calc}
\usepackage{systeme}
\usepackage{amsfonts}
\usepackage{amsthm}
\usepackage{algorithm2e}
\usepackage[applemac]{inputenc}
\usepackage{bbold}
\usepackage[numbers]{natbib}

\numberwithin{equation}{section}
\newcommand{\R}{\mathbb{R}}
\newcommand{\be}{\begin{equation}}
\newcommand{\ee}{\end{equation}}
\newcommand{\ba}{\begin{aligned}}
\newcommand{\ea}{\end{aligned}}
\newcommand{\eps}{\epsilon}
\newcommand{\Pt}{\mathcal{P}_2}
\newcommand{\PR}{\mathcal{P}(\R)}

\newtheorem{Theorem}{Theorem}[section]
\newtheorem{Lemma}[Theorem]{Lemma}
\newtheorem{Rem}[Theorem]{Remark}
\newtheorem{Def}[Theorem]{Definition}
\newtheorem{Prop}[Theorem]{Proposition}
\newtheorem{Cor}[Theorem]{Corollary}

\title{A spectral dominance approach to large random matrices: part II}

\begin{document}





\author{Charles Bertucci \textsuperscript{a}, Jean-Michel Lasry\textsuperscript{b}, Pierre-Louis Lions\textsuperscript{b,c}
}
\address{ \textsuperscript{a} CMAP, Ecole Polytechnique, UMR 7641, 91120 Palaiseau, France,\\ \textsuperscript{b} Universit\'e Paris-Dauphine, PSL Research University,UMR 7534, CEREMADE, 75016 Paris, France,\\ \textsuperscript{b}Coll\`ege de France, 3 rue d'Ulm, 75005, Paris, France. }

\maketitle

\begin{abstract}This paper is the second of a series devoted to the study of the dynamics of the spectrum of large random matrices. We study general extensions of the partial differential equation arising to characterize the limit spectral measure of the Dyson Brownian motion. We provide a regularizing result for those generalizations. We also show that several results of part I extend to cases in which there is no spectral dominance property. We then provide several modeling extensions of such models as well as several identities for the Dyson Brownian motion.

\end{abstract}




\setcounter{tocdepth}{1}

\tableofcontents
\section{Introduction}
In this paper, we precise and extend some results and proofs introduced in part I \citep{bertucci2022spectral}. Our main interest is the study of the partial differential equation (PDE) 

\be\label{dyson}
\ba
\partial_t m + \partial_x (mH[m]) = 0 \text{ in } (0,\infty)\times \mathbb{R},\\
m|_{t = 0} \in \mathcal{P}(\R),
\ea
\ee
and of some of its extensions. We shall call \eqref{dyson} the Dyson equation in the rest of the paper. In \eqref{dyson}, $H[m]$ denotes the Hilbert transform of $m$ and $\mathcal{P}(\R)$ denotes the set of probability measures on $\R$. This equation naturally arises in the theory of random matrices as the limit equation for the spectral measure of the Dyson Brownian motion. We refer to \citep{anderson2010introduction} for a comprehensive and detailed introduction to the theory of random matrices.

For the moment, most of the literature on random matrices on equations of the form of \eqref{dyson} has been focused on the exact nonlinearity $mH[m]$. Following part I, we present results on other nonlinearities of the same form as well as for equations involving general drift terms. Let us insist upon the fact that, for practical applications, it is fundamental to develop techniques which are robust to small changes in \eqref{dyson}.\\

The rest of the paper is organized as follows. In Section \ref{sec:notation}, we recall notations and results of the first part which shall be used here. We then precise and correct some results and proofs of the first part I in Section \ref{sec:cor}. In Section \ref{sec:reg}, we provide a regularizing result for this class of equations. In Section \ref{sec:extension}, we extend the results of part I to more general equations, namely ones without the spectral dominance. We also give several identities on the pure equation \eqref{dyson} in Section \ref{sec:identities}. We then conclude this paper by indicating some extensions which are natural from a modeling perspective in Section \ref{sec:model}.


\section{Notation and results of the first part}\label{sec:notation}
\subsection{Notation}
The set of Borel probability measures on $\R$ is $\mathcal{P}(\R)$, when not specified otherwise, it will be endowed with the weak $*$ convergence.  For any $1 \leq p \leq \infty$, $u \in L^p(\R)$, we denote by $H[u]$ its Hilbert transform, that is the linear anti-self adjoint operator on $L^p$ defined on smooth function with the formula
\[
H[u](x) = \int_\R \frac{u(y)-u(x)}{x-y}dy.
\]
Recall that for $m \in \mathcal{P}(\R)$, $H[m]$ is the distribution $p.v.(\frac 1x)*m$, defined against smooth test function $\phi$ through
\[
\left\langle p.v.(\frac 1x)*m,\phi \right\rangle = - \int_\R H[\phi] dm.
\]
We denote by $A_0$ the half-Laplacian operator, defined by $A_0 = \frac{d}{dx}H$. On smooth functions, it is thus given by
\[
A_0[u](x) = \int_{\R}\frac{u(x)-u(y)}{(x-y)^2}dy.
\]
Let us also define for $\delta > 0$ the operators $A_{\delta}$ and $A_{-\delta}$ with
\[
A_\delta[u](x) = \int_{|x-y| \geq \delta}\frac{u(x)-u(y)}{(x-y)^2}dy, \quad A_{-\delta}[u](x) = \int_{|x-y| < \delta}\frac{u(x)-u(y)}{(x-y)^2}dy.
\]
For a locally bounded function $u$ on a subset $\Omega$ of $\R^d$, we define for all $x \in \Omega$
\[
u_*(x) = \liminf_{y \to x} u(y), \quad u^*(x) = \limsup_{y \to x}u(y).
\]

\subsection{Characterization of solutions of the Dyson equation}
As in part I, we are going to going to say that $m$ is a solution of \eqref{dyson} if $u(t,x) := m_t((-\infty,x])$ is a viscosity solution of the primitive PDE
\be\label{dysonint}
\partial_t u + \partial_x u A_0[u] = 0 \text{ in } (0,\infty)\times \R,
\ee
which is simply obtained by integrating \eqref{dyson} with respect to $x$. Naturally, the previous equation is associated with boundary conditions 
\[
\forall t \geq 0, \quad u(t,-\infty) = 0, \quad u(t,\infty) = 1,
\]
and the initial condition $u_0 := m_0((-\infty,x])$, which is thus non-decreasing.
We showed in part I that, in this context, this equation is well suited for the theory of viscosity solutions since it has a comparison principle. Moreover, this comparison principle does not rely on the exact form of $A_0$ (hence of $H[\cdot]$), but rather on the fact that it is an elliptic operator. Hence, it extends naturally to more general equations, namely PDE of the form
\be\label{dyson:gen}
\partial_t u + \partial_x u c(x)A_0[u] + b \partial_x u = 0 \text{ in } (0,\infty)\times \R,
\ee
where $c:\R\to \R$ is a non-negative continuous function and $b: \R \to \R$ is continuous. We now give a precise definition of viscosity solutions for \eqref{dyson:gen}.
\begin{Def}
An usc (resp. lsc) function $u : \R_+\times \R$ is a viscosity sub-solution of \eqref{dyson:gen} (resp. a super-solution) if for any $T > 0$, for any $\phi \in C^{1,1}_b$, $0 < \delta \leq \infty$  $(t_0,x_0) \in (0,T]\times \R$ point of maximum (resp. minimum) of $u-\phi$,
\[
\ba
\partial_t& \phi(t_0,x_0)  + \partial_x\phi(t_0,x_0)c(x_0) \left( A_{-\delta}[\phi(t_0)](x_0) + A_{\delta}[u(t_0)](x_0)\right) \\
&+ b(x_0)\partial_x \phi(t_0,x_0)\leq 0 \text{ (resp. } \geq 0).
\ea
\]
\end{Def}
\begin{Def}
A viscosity solution of \eqref{dyson:gen} is a locally bounded function such that $u^*$ is a viscosity super-solution and $u_*$ is viscosity sub-solution of \eqref{dyson:gen}.
\end{Def}
One of the main result of part I was to establish a comparison principle (Proposition 4.4 in \citep{bertucci2022spectral}) for equations of the form of \eqref{dysonint}, which was called a spectral dominance property because of the interpretation of the equation in terms of spectral measure.

Below we shall also use the notion of viscosity solutions for slightly more involved PDE and in those cases, the adequate notion of viscosity solution is the natural extension of this one.
In part I, several results were introduced on equations of the form of \eqref{dyson:gen}, namely the propagation of Lipschitz estimates as well as comparison results for viscosity solutions.

Let us remark that equations of the same form as \eqref{dyson:gen} have been considered independently of random matrices, namely because of their applications in dislocations problems. We refer the reader to \citep{imbert2008homogenization,imbert2008homogenization2,forcadel2009homogenization} for more detailed and such equations and their uses.

\subsection{Derivation of the Dyson equation}
We briefly recall the usual derivation of \eqref{dyson}, using the so-called Dyson Brownian motion.
Consider a collection $(B^i)_{i \in \mathbb{N}}$ of independent Brownian motions and a sequence $(\lambda^N)_{N \in \mathbb{N}}$ such that for all $N \geq 1, \lambda^N \in \R^N$ with $\lambda_i^N < \lambda_j^N$ as soon as $i < j$. It is well known (see e.g. \citep{chan,rogers,anderson2010introduction}) that when, in law,
\[
 \lim_{N \to \infty}N^{-1}\sum_{i =1}^N \delta_{\lambda^N_i} = m_0 \in \mathcal{P}(\R),
 \]
 then for all $t > 0$, $N^{-1}\sum_{i =1}^N \delta_{\lambda^N_{i,t}}$ converges in law toward $m_t \in \mathcal{P}(\R)$, where the processes $((\lambda^N_{i,t})_{t \geq 0})_{1\leq i \leq N}$ are the unique strong solutions of
\be\label{dysonN}
d\lambda^N_{i,t} = \frac1N\sum_{j \ne i} \frac{1}{\lambda^N_{i,t}-\lambda^N_{j,t}}dt + \frac{\sqrt{2}}{\sqrt{N}}dB^i_t
\ee
and $(m_t)_{t \geq 0} \in C(\R_+,\mathcal{P}(\R))$ is the unique solution of the Dyson equation \eqref{dyson} with initial condition $m_0$.

\section{Corrections on the first part}\label{sec:cor}
The first correction we make is on the convergence result presented in Theorem $3$ in \citep{bertucci2022spectral}. At the beginning of the proof, we considered an element $F$ given as a limit point of the sequence $(F_N)_{N \geq 2}$ through a compactness argument. If this sequence is indeed relatively-compact almost surely, the problem is that the limit point depends on $\omega \in \Omega$. The exact same argument can be made correct by simply considering instead
\[
F^*(t,x) = \limsup_{N \to \infty, x_N \to x, t_N \to t} F_N(t_N,x_N).
\]
Note that the previous object is a priori stochastic, because so is the system of SDE. The exact same computations we presented in \citep{bertucci2022spectral} then establishes that, almost surely, $F^*$ is a viscosity sub-solution of the equation. The same argument can be made to show that $F_*$, the $\liminf$ of $(F_N)_{N \geq 2}$, is almost surely a viscosity supersolution. We conclude as usual in viscosity solution theory, using the comparison result (Proposition 4.4 in \citep{bertucci2022spectral}). Indeed, it establishes that almost surely $F^* \leq F_*$, which implies of course the equality and thus the existence of a viscosity solution.\\

Corollary $2.2$ of \citep{bertucci2022spectral} is proved and detailed in Section 4.4 below.\\

The equations (4.31) and (4.32) of \citep{bertucci2022spectral} should read differently, as a negative sign should be in front of the constant $C$. In (2.43) and (4.52), it lacks the condition $ x \geq y$.

\section{A regularizing result}\label{sec:reg}
In \citep{biane1997free}, the author established strong regularizing properties of the pure Dyson equation \eqref{dyson}. Namely he showed that given an initial condition $m_0 \in \mathcal{P}(\R)$, the free convolution of $m_0$ with the semi-circular distribution, which turns out to be the unique solution of \eqref{dyson}, is bounded and $\frac 13$ H\"older continuous. He showed that for any $t > 0,\|m_t\|_\infty \leq \frac{1}{\sqrt t}$. This results relied exclusively on explicit computations involving the semi-circular distribution.

In this section, we provide another proof of the $L^\infty$ regularizing result, but which is valid for a wider class of equations, namely ones involving drifts and more general interactions.

We are here concerned with solutions of the equation
\be\label{dysong}
\partial_t m + \partial_x(mK[m]) + \partial_x (b m) = 0 \text{ in } (0,\infty)\times \R,
\ee
with initial condition $m_0 \in \mathcal{P}(\R)$, where $b : \R \to \R$ is a given Lipschitz continuous and bounded function and $K$ is a singular integral operator given by
\[
K[m] = \int_\R \frac{f(x,y)}{x-y}m(dy),
\]
for $f : \R^2 \to \R$. This equation is obtained as the natural analogue of \eqref{dyson} when the pairwise interaction in \eqref{dysonN} is given by $f(\lambda_i,\lambda_j)/(\lambda_i - \lambda_j)$. Such interactions can appear in random matrix theory as was shown in part I. We start by giving the main a priori estimate behind the regularizing property. 

\begin{Prop}\label{prop:regsmooth}
Consider $m$ a smooth Lipschitz solution of \eqref{dysong} such that $m_0 \in \PR$. Assume that there exists $C_0 >0$ such that
\be\label{hypf}
\inf_x f(x,x) \geq C_0^{-1}, \quad \|f\|_\infty + \|\partial_1 f \|_\infty + \|\partial_{12}f\|_\infty \leq C_0.
\ee 
Then, for any $T > 0$, there exists $C > 0$ depending only $C_0$, $b$ and $T$ such that 
\be\label{eq:reg}
\forall 0 < t \leq T, \quad m(t,x) \leq \frac{C}{\sqrt{t}} .
\ee
\end{Prop}
\begin{proof}
For any $t > 0$, consider a point $\bar x$ of maximum of $m(t,\cdot)$. From the regularity of $m$, such a point exists. Evaluating \eqref{dysong} at this point, we obtain that
\[
\partial_tm(t,\bar x) + m(t,\bar{x})\partial_x(K[m(t)])(\bar x) + \partial_x b(t,\bar x) m(t,\bar x) = 0.
\]
We now compute for a smooth function $\phi$
\be\label{expression}
\ba
\partial_x K[\phi](x) &= \partial_x\left[ \int_\R \frac{f(x,y)-f(x,x)}{x-y}\phi(y)dy + f(x,x) \int_\R \frac{\phi(y)}{x-y}dy  \right]\\
&= \int_\R \partial_x\left[ \frac{f(x,y)-f(x,x)}{x-y}\right]\phi(y)dy + \partial_x (f(x,x)) \int_\R \frac{\phi(y)}{x-y}dy\\
& \quad + f(x,x) \int_\R \frac{\phi(x) -\phi(y)}{(x-y)^2}dy.
\ea
\ee
Defining $c(x) = f(x,x)$, and $g(x,y) = (f(x,y)-f(x,x))(x-y)$, we now obtain
\[
\ba
\partial_x K[\phi](x) &= \int_\R \partial_x g(x,y) \phi(y)dy + \int_\R (c(x) - c'(x)(x-y))\frac{\phi(x) - \phi(y)}{(x-y)^2} dy.
\ea
\]
Observe that there exists $\kappa, \delta_0 > 0$ depending only on $C_0$ such that
\[
|x-y|\leq \delta_0 \Rightarrow c(x) - c'(x) (y-x) \geq \kappa,
\]
as well as $\|\partial_x g\|_\infty \leq \kappa^{-1}$. Let us now compute for $\delta \leq \delta_0$, since $\bar x$ is a point of maximum of $m$,
\[
\ba
\partial_x(K[m(t)])(\bar x) &\geq - \kappa^{-1} + \kappa \int_{\delta \leq |x-y|\leq \delta_0} \frac{m(t,\bar x) - m(t,y)}{(\bar x-y)^2} dy\\
& \quad  + \int_{ |x-y| > \delta_0}(c(\bar x)- c'(x)(x-y)  ) \frac{m(t,\bar x) - m(t,y)}{(\bar x-y)^2} dy \\
&\geq 2 \kappa (\delta^{-1} - \delta_0^{-1})m(t,\bar x) - \kappa^{-1}- \kappa \int_{\delta \leq |x-y|\leq \delta_0}\frac{m(t,y)}{(\bar x-y)^2} dy\\
& \quad - \int_{|x-y|> \delta_0} (c(\bar x)- c'(x)(x-y)  )\frac{m(t,y)}{(\bar x -y)^2} dy\\
& \geq 2 \kappa (\delta^{-1}- \delta_0^{-1})m(t,\bar x) - \kappa^{-1} - \kappa \delta^{-2} - C,
\ea
\]
where $C$ is a constant which bounds $\frac{c( x)- c'(x)(x-y) }{(x-y)^2}$ on $|x-y|> \delta_0$. Hence, coming back to the equation satisfied by $m$, we deduce that
\[
\partial_t m(t,\bar x) \leq - m(t,\bar x) (- \|\partial_x b\|_\infty - C - \kappa \delta^{-2} + 2\kappa (\delta^{-1} - \delta_0^{-1})m(t,\bar x)).
\]
Thus, choosing $\delta = (m(t,\bar x))^{-1}\land \delta_0$, we obtain that as soon as $m(t,\bar x) > \delta_0^{-1}$
\[
\partial_t m(t,\bar x) \leq - m(t,\bar x) ( \kappa m(t,\bar x)^2 - C).
\]
Hence, the claim follows.

\end{proof}
\begin{Rem}
This result is indeed a regularizing result since the constant $C$ does not depend on $m_0$.
\end{Rem}
In the previous result, $C$ does not depend on $m_0$. Hence, we would like to generalize this result to non-smooth $m_0$. We now continue in this direction. Remark that the previous proof can be immediately extended to the case of the equation
\be\label{dysongreg}
\partial_t m - \eps \partial_{xx} m+ \partial_x(mK[m]) + \partial_x (b m) = 0 \text{ in } (0,\infty)\times \R,
\ee
for $\eps > 0$. Indeed we trivially have the following.
\begin{Cor}
For $\eps > 0$, consider $m$ a smooth Lipschitz solution of \eqref{dysongreg} such that $m_0 \in \PR$. Then, under the assumptions of Theorem \ref{thm:reg}, for any $T > 0$, there exists $C > 0$ depending only $C_0$, $b$ and $T$ (but not on $\eps$) such that \eqref{eq:reg} holds.
\end{Cor}

Our strategy for proving a regularizing result for general solutions is the following. We are going to show that there always exists a solution $m_\eps$ of \eqref{dysongreg}, which satisfies the required estimate thanks to the previous Corollary. Then, by passing to the limit $\epsilon \to 0$, we are going to show that there exists a limit point of $(m_\eps)_{\eps}$ which is a solution of \eqref{dysong} which satisfies \eqref{eq:reg}. 

We start by proving an existence result of the regularized equation.
\begin{Prop}\label{prop:reg2}
Let $m_0 \in \PR$ have a smooth density $m_0$. Assume that $m_0$, $f$ and $b$ are $C^\infty$ functions with all their derivatives bounded, as well as $\inf_x f(x,x) > -\infty$.  Then, there exists a unique smooth solution $m_\eps$ of \eqref{dysongreg}. Moreover, it satisfies \eqref{eq:reg} for a constant $C$ as in Theorem \ref{thm:reg}, in particular, independent of $\eps$.
\end{Prop}
\begin{proof}
Fix a time horizon $T > 0$. We start by assuming that such a smooth function exists and we want to show an a priori estimate (depending on $\eps > 0$ obviously). We denote by $C> 0$ a constant which depends only on $m_0$ and $\epsilon$. First, since the term in $\epsilon$ does not perturb the proof of the $L^{\infty}$ a priori estimate above, there exists $C > 0$ such that $\|m\|_{\infty} \leq C$. Since $m(t,\cdot)$ is in $L^1$ for all time $t \geq 0$, we also deduce that it is in all the $L^p$, for $1 \leq p \leq \infty$. Multiplying \eqref{dysongreg} by $m$ and integrating over space and time, we then obtain
\[
\frac{d}{dt}\int_\R m^2 + \epsilon \int_0^t\int_\R (\partial_x m)^2 = \int_0^t\int_\R (\partial_x m)   (m K[m] + mb).
\]
Furthermore, by simply splitting the integral defining $K[m]$ around the singularity, we also have that for all $\delta > 0$, there exists $C_\delta > 0$ such that for all $µ$, $\|K[µ]\|_\infty \leq \delta \|µ\|_{C^{\alpha}} + C_\delta \|µ\|_{L^1}$, for $\alpha \in (0,\frac 12)$. Hence, recalling the energy estimate, we obtain
\[
\frac{d}{dt}\int_\R m^2 + \epsilon \int_0^t\int_\R (\partial_x m)^2 \leq \|m\|_{L^2(H^1)}(\kappa \|m\|_{L^2(C^\alpha)} + C_\kappa \|m_0\|_{L^1}) + C,
\]
where we have used the fact that $b$ is Lipschitz continuous.
From the usual embedding of $H^1$ into $C^\alpha$, we deduce that $\|m\|_{L^2(H^1)} \leq C$.\\

Further regularity can be obtained by looking at the equation satisfied by $w = \partial_x m$. Multiplying this equation by $w$ and integrating yields
\[
\frac{d}{dt} \int_\R w^2 + \epsilon \int_0^t \int_\R (\partial_x w)^2 = \int_0^t \int_\R \partial_x w(w K[m] + m \partial_x K[m]) - \int_0^t \int_\R \partial_{xx} b mw - \frac 32\partial_x b w^2.
\]
Using the expression of $\partial_x K [m]$ derived in \eqref{expression}, and arguing as in the previous case, we deduce that $\|w\|_{L^\infty(L^2)} + \|w\|_{L^2(H^1)} \leq C$. Further regularity can then be obtained in a similar fashion by bootstrapping techniques.\\

Once regularity is obtained, the existence of a solution of \eqref{dysongreg} with this regularity is classical and we do not detail it here. Furthermore, in this situation, the uniqueness of such a solution is immediate. Finally, the fact that it satisfies the required $L^\infty$ estimate follows simply from the previous result.
\end{proof}

Now, we would like to pass to the limit $\epsilon \to 0$ in the previous equation to consider general initial conditions and solutions of \eqref{dysong}. Recall that, in general, we say that $m$ is a solution of \eqref{dysong} if the fonction $u$ defined by $u(t,x) = \int_{-\infty}^xm_t(dy)$ is a viscosity solution of the associated equation
\be\label{eq:prim}
\begin{aligned}
\partial_t u + \partial_x u L[u] + b \partial_x u = 0 \text{ in } (0,\infty)\times \R,
u(0,x) = m_0((-\infty,x]) \text{ in } \R.
\end{aligned}
\ee
where $L$ defined on smooth function $\phi$ by
\[
L[\phi](x) = \int_\R g(x,y) \frac{\phi(x)-\phi(y)}{(x-y)^2}dy,
\]
where $g(x,y) = f(x,y) + (x-y)\partial_y f(x,y) $.
In order to use the stability of viscosity solutions, we need a comparison principle. Hence, we are going to assume that $f$ is such that a comparison principle holds, that is
\be\label{hyp:fcomp}
\forall x \in \R, y \to \frac{f(x,y)}{x-y} \text{ is non-incresaing on } (-\infty,x) \text{ and on } (x,\infty),
\ee
which resumes here to assuming that $g \geq 0$. We can now state the following.

\begin{Prop}\label{prop:approx}
Consider three sequences $(m_{0,\eps})_{\eps > 0}$, $(f_\eps)_{\eps> 0}$ and $(b_{\eps})_{\eps > 0}$, valued in respectively $\PR$, $C^1(\R^2,\R)$ and $C^1(\R,\R)$. Denote $g_\eps(x,y) = f_\eps(x,y) + (x-y)\partial_y f_\eps(x,y)$. Assume that
\begin{itemize}
\item for all $\epsilon > 0$, $m_{0,\eps}$, $f_\eps$ and $b_\eps$ are $C^\infty$ with all derivatives bounded.
\item $(m_{0,\eps})_{\eps}$ and $(b_\eps)_{\eps}$ converge locally uniformly toward some $m_0$ and $b$.
\item For all $\eps > 0$, $g_\eps \geq 0$.
\item $(g_\eps)_{\eps > 0}$ and $(f_\eps)_{\eps > 0}$ converge uniformly toward $g$ and $f$.
\item the assumptions of Theorem \ref{thm:reg} are satisfied uniformly in $\epsilon$.
\end{itemize}
Consider a sequence $(u_\eps)_{\eps}$ of viscosity sub-solution (resp. super-solution) of 
\[
\begin{aligned}
\partial_t u_\eps -\eps \partial_xx u_\eps + \partial_x u_\eps L_\eps[u_\eps] + b_\eps \partial_x u_\eps = 0 \text{ in } (0,\infty)\times \R,
u_\eps(0,x) = m_{0,\eps}((-\infty,x]) \text{ in } \R,
\end{aligned}
\]
where $L_\eps$ is defined as $L$ when $f$ is replaced by $f_\eps$. Then, if $(u_\eps)_{\eps > 0}$ is uniformly bounded from above (resp. from below), the function $\bar u$ defined by
\[
\bar u(t,x) = \limsup_{\eps \to 0,s \to t, y \to x} u_\eps(s,y) \text{ ( resp. }\underbar u = \liminf\text{)}
\]
 is a viscosity sub-solution (resp. super-solution) of \eqref{eq:prim}.
\end{Prop}

\begin{proof}
Consider $\phi$ a smooth function, $T>0$ and $(t_0,x_0) \in (0,T]\times \R$ a point of maximum of $\bar u - \phi$ on $[0,T]\times \R$. Without loss of generality, we can assume that it is a point of strict maximum. Denote by $(t_\eps,x_\eps)$ a point of maximum of $u_\eps - \phi$. Note that $(t_\eps,x_\eps) \to (t_0,x_0)$ as $\eps \to 0$. Remark that, up to changing $\phi$ far from $(t_0,x_0)$, we can always assume that such a point exists. Since $u_\eps$ is a viscosity sub-solution of \eqref{dysongreg}, it follows that for all $\delta > 0$
\[
\begin{aligned}
\partial_t& \phi(t_\eps,x_\eps) - \eps \partial_{xx}\phi(t_\eps,x_\eps) + \partial_x \phi(t_\eps,x_\eps)\int_{|x_\eps - y|\leq \delta} g_\eps(x_\eps,y)\frac{\phi(t_\eps,x_\eps) - \phi(t_\eps,y)}{(x_\eps-y)^2} dy\\
 &+  \partial_x \phi(t_\eps,x_\eps)\int_{|x_\eps - y|> \delta} g_\eps(x_\eps,y)\frac{u_\eps(t_\eps,x_\eps) - u_\eps(t_\eps,y)}{(x_\eps-y)^2} dy \leq 0.
 \end{aligned}
\]
Remark that $\partial_x \phi(t_\eps,x_\eps) \geq 0$ and recall that $g_\eps \geq 0$. Then, using Fatou's Lemma, we can pass to the limit in the last integral and we obtain the required viscosity solution formulation. Thus $\bar u$ is indeed a viscosity sub-solution. A similar argument holds in the case of viscosity super-solution.
\end{proof}

We are now ready to prove the main result of this section.
\begin{Theorem}\label{thm:reg}
Assume that \eqref{hypf} holds for some $C_0 > 0$ and that \eqref{hyp:fcomp} holds. Then, for any $T >0$, there exists $C > 0$ depending only $C_0$, $b$ and $T$ such that the unique solution $m$ of \eqref{dysong} is bounded for any $t > 0$ and satisfies for all $t\in (0,T]$
\be\label{eq:reg}
m(t,x) \leq \frac{C}{\sqrt{t}} .
\ee
\end{Theorem}
\begin{proof}
Assume first that $m_0$ is smooth. Consider three sequences $(m_{0,\epsilon})_{\eps > 0}$, $(f_\eps)_{\eps > 0}$ and $(b_\eps)_{\eps> 0}$ satisfying the assumptions of Proposition \ref{prop:approx} and $u_\eps$ the associated solution of \eqref{dysongreg}. Since $m_0$ is smooth, we deduce that for all $x \in \R$,
\[
 \limsup_{\eps \to 0,s \to 0, y \to x} u_\eps(s,y) =  \liminf_{\eps \to 0,s \to 0, y \to x} u_\eps(s,y).
\]
Hence, we deduce from Proposition \ref{prop:approx} and the comparison principle in part I that $u_{\epsilon}$ converges almost everywhere pointwise toward a function $u$. From Proposition \ref{prop:reg2}, we obtain that, since $(u_\eps)_{\eps > 0}$ is such that $(\partial_x u_\eps)_{\eps > 0}$ satisfies uniformly \eqref{eq:reg}, then so does $\partial_x u$. Hence, the result is proven for smooth initial condition $m_0$. 

Consider now a general $m_0$. Arguing as in part I, we now obtain that the fact that there exists a viscosity solution of \eqref{eq:prim} which can be uniformly approximated from above or from below by functions satisfying the required Lipschitz estimate. Hence the result is also valid for general initial conditions.
\end{proof}

Note that we have here used the positivity of the function $g$ in order to prove the regularizing effect for general solutions whereas it was not a requirement of the proof of Proposition \ref{prop:regsmooth}. We develop the case of more generals $g$ in the next Section.

\section{Other extensions of the Dyson equation}\label{sec:extension}
In this section, we discuss various natural extensions of \eqref{eq:prim}. The first one is a case in which the comparison principle does not hold anymore. The second one is the case in which the diagonal term in the non-local interaction (the term $c$) vanishes. The third one is the case in which the drift term $b$ is more singular than Lipschitz continuous.

\subsection{A stability result for non-parabolic equations}
We consider here equations of the form
\be\label{eq143}
\partial_t u + \partial_x u c(x) A_0[u] + B(x;u)\partial_x u = 0 \text{ in } (0,\infty)\times \R,
\ee
where  $c : \R \to (0,\infty)$ is a smooth function such that $\inf_x c(x) > 0$ and $B$ is given by
\[
B(x;u) = \int_\R \beta(x,y)u(y)dy,
\]
for a smooth $\beta : \R^2 \to \R$ which is not necessary non-negative. Recall that this type of equation models large interacting systems, similar to \eqref{dysonN}, in which the pairwise interaction is given by $\frac{f(\lambda_i,\lambda_j)}{\lambda_i -\lambda_j}$. In such a case, we obtain that $c(x) = f(x,x)$ and 
\[
\beta(x,y) = \frac{f(x,y) + (x-y)\partial_y f(x,y) - f(x,x)}{(x-y)^2},
\]
which is smooth as soon as $f$ is smooth. In \citep{bertucci2022spectral}, we explained in details why the positivity of $\beta$ implies that \eqref{eq143} has a comparison principle. We now investigate in details the case in which such a comparison does not hold. We make the following assumptions on $\beta$. The requirements are stated with the convention that $x$ is the first argument of $\beta$ and $y$ its second. Furthermore, the operator $A_0$ is always used on the $x$ first argument. We assume that $\beta$ is smooth and that
\be\label{Bdef}
\beta, \partial_x \beta, \partial_xx \beta \in L^\infty_x(L^1_y).
\ee
\be\label{Bcomp}
A_0 \beta, A_0 \partial_x \beta, A_0\partial_{xx}\beta \in L^1_{x,y}.
\ee
The first condition \eqref{Bdef} ensures that $B(\cdot;u)$, $\partial_x B(\cdot;u)$ and $\partial_x B(\cdot;u)$ are bounded functions as soon as $u \in L^\infty$. The second condition implies that all those functions have a bounded $\dot H^\frac12$ semi-norm depending only on $\|u\|_\infty$. This last fact simply follows from \eqref{12A0}.

We prove the following result which yields uniqueness and stability in $L^\infty$ thanks to a Gronwall like estimate.

\begin{Theorem}\label{thm:gronwallvisc}
Assume that \eqref{Bdef} holds. Consider $T>0$, $u_1$ and $u_2$ two viscosity solutions of 
\[
\partial_t u + c(x)\partial_xu A_0[u] + B(x; u)\partial_x u = 0 \text{ in } (0,T)\times \R,
\]
such that $\partial_x u_1,\partial_x u_2 \in L^\infty$ and such that $u_1$ and $u_2$ are both continuous, bounded and satisfy $u_1(0,x) = u_2(0,x)$ for all $x \in \R$. Then $u_1 = u_2$.
\end{Theorem}
\begin{proof}
Let us first remark that \eqref{Bdef} is clearly sufficient to show that there exists $C>0$ such that for $v,w \in L^\infty$, 
\[
\|B(\cdot;v) - B(\cdot;w)\|_\infty \leq C \|v-w\|_\infty.
\]
 We denote by $M(t) := \|u_2(t) -u_1(t)\|_\infty$ and, for $\epsilon,r > 0$, we consider 
\[
M_{\epsilon,r}(s,t,x,y) := u_1(t,x) - u_2(s,y) - \frac{1}{2\epsilon}((x-y)^2 + (t-s)^2) - r\psi(x),
\]
where $\psi(x) = (x^2+1)^\frac13$. We want to show that $M(t) = 0$ for all $t \geq 0$. In order to do so, we are going to prove that $M(t) \leq C\int_0^tM(s)ds$ for some $C > 0$, which is sufficient to establish the claim since $M(0) = 0$. Recall that from the time regularity of $u_1$ and $u_2$, $M(\cdot)$ is continuous. Assume that the previous inequality does not hold, hence, up to exchanging $u_1$ and $u_2$, there exists $\kappa,\gamma,T > 0 > 0$ such that for any $\epsilon,r > 0$ sufficiently small, 
\[
\sup\left\{ M_{\epsilon,r} - \gamma(t+s) - C_0 \int_0^tM(s)ds | x,y \in \R, t,s \in [0,T]\right\} \geq \kappa,
\]
where $C_0$ is a constant to be chosen later on. Since $u_1$ and $u_2$ are bounded, there exists a point of maximum $(t_*,s_*,x_*,y_*)$ of $M_{\epsilon,r}(t,s,x,y) - \gamma (t+s)-C_0\int_0^tM(s)ds$. We are going to assume that $s_*,t_* > 0$. The case in which one of the two is $0$ can be treated by standard arguments quite easily.

Since $u_1$ is a viscosity solution of \eqref{eq143}, it is in particular a viscosity sub-solution of this equation. Hence, we obtain that for any $\delta > 0$
\[
\ba
c(x_*)\left(\frac{x_*-y_*}{\epsilon} + r\psi'(x_*)\right)\left(\int_{\{|z|\leq \delta\}}\frac{\varphi_1(x_*) -\varphi_1(x_*+z)}{z^2}dz + \int_{\{|z | > \delta\}} \frac{u_1(t_*,x_*) -u_1(t_*,x_* + z)}{z^2}  dz\right)\\
+ \frac{t_*-s_*}{\epsilon} + \gamma + C_0 M(t_*)+ B(x_*;u_1(t_*))\frac{x_*-y_*}{\epsilon} \leq 0,
\ea 
\]
where $\varphi_1(x) = \frac{1}{2\epsilon}(x-y_*)^2 + r \psi(x)$. Similarly, since $u_2$ is a viscosity super-solution of \eqref{eq143}, we also deduce that
\[
\ba
c(y_*)\frac{x_*-y_*}{\epsilon} \left(\int_{\{|z|\leq \delta\}}\frac{\varphi_2(y_*) -\varphi_2(y_*+z)}{z^2}dz + \int_{\{|z | > \delta\}} \frac{u_2(s_*,y_*) -u_2(s_*,y_* + z)}{z^2}  dz\right)\\
+ \frac{t_*-s_*}{\epsilon} - \gamma + B(y_*;u_2(s_*))\frac{x_*-y_*}{\epsilon} \geq 0,
\ea 
\]
where $\varphi_2(y) = -\frac{1}{2\epsilon}(x_*-y)^2$. Combining the two relations yields
\[
\ba
2& \gamma + c(y_*)\frac{x_*-y_*}{\epsilon}\bigg(\int_{\{|z|\leq \delta\}}\frac{\varphi_1(x_*) - \varphi_2(y_*)-\varphi_1(x_*+z) + \varphi_2(y_*+z)}{z^2}dz\\
&\quad \quad \quad + \int_{\{|z | > \delta\}} \frac{u_1(t_*,x_*) -u_2(s_*,y_*)-u_1(t_*,x_* + z)+u_2(s_*,y_*+z)}{z^2}  dz\bigg)\\
&+\left(r \psi'(x_*)c(x_*) + \frac{(c(x_*)-c(y_*))(x_*-y_*)}{\epsilon}\right)\bigg(\int_{\{|z|\leq \delta\}}\frac{\varphi_1(x_*) -\varphi_1(x_*+z)}{z^2}dz\\
& \quad \quad \quad \quad \quad\quad \quad \quad \quad \quad \quad \quad\quad \quad \quad \quad \quad \quad \quad \quad+ \int_{\{|z | > \delta\}} \frac{u_1(t_*,x_*) -u_1(t_*,x_* + z)}{z^2}  dz\bigg)\\
&+ \frac{x_*-y_*}{\epsilon}(B(x_*; u_1(t_*))-B(y_* ; u_2(s_*)))  + C_0M(t_*)\leq 0
\ea
\]
Using the definition of $\varphi_1$ and $\varphi_2$, we obtain that
\[
\int_{\{|z|\leq \delta\}}\frac{\varphi_1(x_*) - \varphi_2(y_*)-\varphi_1(x_*+z) + \varphi_2(y_*+z)}{z^2}dz \leq C\delta(r + \epsilon^{-1}).
\]
From the definition of $(t_*,s_*,x_*,y_*)$ we deduce that
\[
\ba
 \int_{\{|z | > \delta\}} \frac{u_1(t_*,x_*) -u_2(s_*,y_*)-u_1(t_*,x_* + z)+u_2(s_*,y_*+z)}{z^2}  dz &\leq  r\int_{\{|z | > \delta\}} \frac{\psi(x_*) - \psi(x_*+z)}{z^2}  dz\\
 &\leq r(1 + \delta^{-1}).
 \ea
\]
From the same argument, it also holds true that
\[
\int_{\{|z|\leq \delta\}}\frac{\varphi_1(x_*) -\varphi_1(x_*+z)}{z^2}dz + \int_{\{|z | > \delta\}} \frac{u_1(t_*,x_*) -u_1(t_*,x_* + z)}{z^2}  dz \leq C\delta(r + \epsilon^{-1}) + C \delta^{-1}
\]
Finally, thanks to the regularity of $B$,
\[
|B(x_*;u_1(t_*))-B(y_*; u_2(s_*))| \leq C_1( |x_*-y_*|+ \|u_1(t_*) - u_2(s_*)\|_\infty).
\]
Using the four previous estimate, we deduce that
\[
\ba
2 \gamma + C_0M(t_*)\leq  &C\left(\frac{|x_*-y_*|}{\epsilon}+ r + \omega(\epsilon)\right)(\delta(r + \epsilon^{-1}) +  r(1 + \delta^{-1}))\\
& +C_1( |x_*-y_*| + \|u_1(t_*) - u_2(s_*)\|_\infty )\left|\frac{x_*-y_*}{\epsilon}\right|.
\ea
\]
As usual in this kind of arguments, we always have $(t_*-s_*)^2 + (x_*-y_*)^2\leq \omega(\epsilon) \epsilon$, where $\omega(\epsilon) \to 0$ as $\epsilon \to 0$. Let us take $K>0$ such that $u_1$ is, uniformly in $t$, $K$ Lipschitz in $x$. For such a $K$, $|x_*-y_*| \leq K\epsilon$. Moreover, using the time regularity of $u_1$, we can compute
\[
\|u_1(t_*) - u_2(s_*)\|_\infty \leq M(t_*) + \|u_2(t_*)-u_2(s_*)\|_\infty \leq M(t_*) +\tilde{\omega}(|t_*-s_*|),
\]
for a modulus of continuity $\omega$, which depends only on $u_2$.  Hence, using those relations, we arrive at
\[
 2 \gamma + C_0M(t_*) \leq C(K + \omega(\epsilon)+ r)(r( 1 + \delta^{-1}) + \delta \epsilon^{-1}) + C_1( K \epsilon+ M(t_*) + C\tilde\omega(\sqrt{\epsilon}))
\]
Hence, taking $C_0 > C_1$ and simplifying by $C_1M(t_*)$, we arrive at a contradiction by taking first the limit $r \to 0$, and then setting $\delta = \epsilon^2$ and taking the limit $\epsilon \to 0$. The previous contradiction implies that $M(t) \leq C_0 \int_0^t M(s)ds$, from which the result follows.
\end{proof}

The previous proof makes apparent a need for the continuity in time of the solution which we are not able to overcome. In the next section, we explain how can such regularity can be obtained.

\subsection{Time regularity of solutions}

In this situation, to overcome the lack of parabolicity of the equation, we have to use some regularity. Regularity with respect to the $x$ variable holds thanks to Proposition \ref{prop:regsmooth}. We explain why we can also provide an estimate on the time derivative, namely by bounding all the other terms in \eqref{eq143}.

Let us recall the usual $\dot H^{\frac{1}{2}}(\R)$ Sobolev semi norm defined by
\be\label{12A0}
\|u\|^2_{\dot H^\frac12} := \int_\R u(x)A_0[u](x)dx.
\ee
The Fourier representation of this semi-norm is
\[
\|u\|^2_{\dot H^\frac12} = \int_\R |\xi| |\hat u(\xi)|^2d\xi.
\]
We have for any $u,v \in H^\frac12$
\be\label{eq:254}
\ba
\int_\R u\partial_xv \leq \|u\|_{\dot H^\frac12}\|v\|_{\dot H^{\frac 12}}, \quad \int_\R u A_0[v] \leq \|u\|_{\dot H^\frac12}\|v\|_{\dot H^{\frac 12}}.
\ea
\ee
We shall use the following (standard) Lemma several times.
\begin{Lemma}\label{lemma:alg}
Let $u$ and $v$ be a bounded functions such that $  \|u\|_{\dot H^\frac12}+\|v\|_{\dot H^\frac12} < \infty$. Then 
\[
\|uv\|^2_{\dot H^\frac12}\leq 2(\|u\|^2_\infty\|v\|^2_{\dot H^\frac12}+ \|u\|^2_{\dot H^\frac12}\|v\|^2_\infty),
\]
where $C$ depends on $v$.
\end{Lemma}
\begin{proof}
It suffices to verify it with the Fourier transform.
\[
\ba
\int_\R |\xi|\left|\int_\R \hat u(\xi -z)\hat v (z)dz\right|^2d\xi &= \int_\R \left|\int_\R \hat u(\xi -z)\sqrt{|\xi|}\hat v (z)dz\right|^2d\xi\\
& \leq \int_\R \left|\int_\R \hat u(\xi -z)(\sqrt{|\xi-z|} + \sqrt{|z|})\hat v (z)dz\right|^2d\xi\\
& \leq \int_\R \left(\left|\int_\R \hat u(\xi -z)\sqrt{|\xi-z|}\hat v(z)dz\right| +\left|\int_\R\hat u(\xi-z)\hat v(z) \sqrt{|z|}dz\right|\right)^2d\xi\\
& \leq2( \| u\|^2_\infty\|u\|^2_{\dot H^\frac12} + \| u\|^2_\infty \|v\|^2_{\dot H^\frac12}).
\ea
\]

\end{proof}

We can prove our estimate on the time regularity of the solution.

\begin{Prop}\label{prop:estlip}
Assume that \eqref{Bdef} and \eqref{Bcomp} hold. Let $u:[0,\infty) \times \R$ be a smooth function solution of \eqref{eq143} such that for all $t \geq 0, u(t,\cdot)$ is non-decreasing, bounded, and $u(-\infty) = 0, u(\infty) =1$. Then, for any $T > 0$, there exists $C > 0$ depending only $T,c,B$ and $\| u(0,\cdot)\|_{C^{1,\alpha}}$ such that $|\partial_t u| \leq C$ on $[0,T]\times \R$.
\end{Prop}
\begin{proof}
\textbf{Step 1: An estimate on the singular term.} We start by multiplying \eqref{eq143} by $A_0[u]$ and by integrating over space, we then obtain
\be\label{esth12}
\ba
\frac{d}{dt}\frac 12& \int_\R u(t,x)A_0[u(t)](x)dx + \int_\R c(x)\partial_xu(t,x)(A_0[u(t)](x))^2dx\\
& = -\int_\R A_0[B(\cdot;u(t))\partial_xu(t,\cdot)](x)u(t,x)dx
\ea
\ee
Let us now remark that
\[
\ba
\int_\R A_0[B(\cdot;u(t))\partial_xu(t,\cdot)](x)u(t,x)dx = &\int_\R \partial_x\left(A_0[u(t) B(\cdot;u(t))](x) \right)u(t,x)dx\\
&- \int_\R u(t,x)A_0[\partial_x B(\cdot;u(t))u(t)](x)dx,
\ea
\]
from which we obtain
\[
\ba
2\int_\R A_0[B(\cdot;u(t))\partial_xu(t,\cdot)](x)u(t,x)dx = & \int_\R A_0[B(\cdot;u(t))\partial_xu(t,\cdot)](x)u(t,x)dx\\
&  + \int_\R \partial_x\left(A_0[u(t) B(\cdot;u(t))](x) \right)u(t,x)dx\\
&- \int_\R u(t,x)A_0[\partial_x B(\cdot;u(t))u(t)](x)dx,\\
 =\int_\R \partial_xu(t,x)&\bigg(A_0[u(t) B(\cdot;u(t))](x)  - B(x,u(t))A_0[u(t)](x)\bigg)dx\\
&- \int_\R u(t,x)A_0[\partial_x B(\cdot;u(t))u(t)](x)dx.
\ea
\]
From Lemma \ref{lemma:alg},
\[
\ba
\left| \int_\R u(t,x)A_0[\partial_x B(\cdot;u(t))u(t)](x)dx \right| &\leq \|\partial_x B u\|_{\dot H^\frac12}\|u\|_{\dot H^\frac12}\\
&\leq C(1 + \|u\|^2_{\dot H^\frac12}).
\ea
\]
Let us now compute
\[
\ba
\Gamma (x) :&= B(x)A_0[u(t)](x) - A_0[Bu](x)\\
&= \int_\R \frac{B(x)(u(x) - u(x+z)) - (B(x)u(x) - B(x+z)u(x+z)) }{z^2} dz\\
& = \int_\R \frac{(B(x+z) - B(x))u(x+z)}{z^2}dz\\
& = \partial_xB(x)\int_\R \frac{u(x+z) - u(x)}{z}dz + \int_\R \Theta(x,z)u(x+z)dz,
\ea
\]
where $\Theta$ is a smooth function, given by the Taylor expansion of $B(\cdot;u)$. In the previous, we lost the dependence of $B$ in $u$ to lighten the notation. The second term of the right side is smooth and has a finite $\dot H^{\frac 12}$ seminorm. The first one is the product of a smooth integrable function and $H[u]$. Since $H$ is a bounded operator in $H^{\frac12}$, we can deduce from the Lemma \ref{lemma:alg} that $\|\Gamma\|^2_{\dot H^\frac12} \leq C(1 + \|u\|_{\dot H^\frac12}^2)$. Hence, returning to \eqref{esth12}, we deduce that
\[
\frac{d}{dt}\frac 12\|u(t)\|^2_{\dot H^\frac12} + \int_\R c(x)\partial_xu(t,x)(A_0[u(t)](x))^2dx \leq C (1 + \|u(t)\|_{\dot H^\frac12}^2).
\]
Thus, from Gr\"onwall's Lemma, we obtain that there exists $C > 0$ such that
\[
\sup_{t \leq T}\|u(t)\|_{\dot H^\frac12} + \int_0^T\int_\R c(x)\partial_xu(t,x)(A_0[u(t)](x))^2dx dt \leq C.
\]

\textbf{Step 2: Estimate on the time derivative of the solution.} Let us remark that since for all $t\geq 0$, $u(t)$ is non-decreasing and bounded, we deduce that $\partial_x u$ in bounded in $L^\infty_t(L^1_x)$. Hence, since $B$ is uniformly bounded, we also deduce that $(t,x)\to B(x;u(t))\partial_xu(t,x)$ is bounded in $L^{\infty}_t(L^1_x)$. On the other hand, we have that for all $t \geq 0, x \in \R$
\[
\left| c(x) \partial_x u(t,x) A_0[u(t)](x)\right| = \sqrt{\partial_xu(t,x)}\sqrt{c(x)\partial_xu(t,x)}|A_0[u(t)](x)|.
\]
Remark now that H\"older's inequality yields that $(t,x) \to c(x) \partial_x u(t,x) A_0[u(t)](x)$ is bounded in $L^2_t(L_x^1)$. Recalling the bound on $B \partial_x u$, we deduce from the equation \eqref{eq143} that $\partial_t u$ is also bounded in $L^2_t(L^1_x)$. Applying $\partial_t$ to \eqref{eq143} and introducing $v = \partial_t u$, we find
\[
\partial_t v + c(x) \partial_x u A_0[v(t)](x) + (c(x)A_0[u(t)] + B(x;u(t)))\partial_x v(t,x) = -B(x,v(t))\partial_xu(t,x) \text{ in } (0,\infty)\times \R.
\]
Hence, from the maximum principle (recall that $A_0$ preserves comparison results and that both $c$ and $\partial_x u$ are non-negative), we obtain that
\[
\frac{d}{dt}\|v(t)\|_\infty \leq \|\partial_x u(t,x)\|_\infty \left|\int_R\beta(x,z)v(t,z)dz\right|.
\]
Then, since $v$ is bounded in $L^2_t(L^1_x)$, it follows that $(t,x) \to B(x,v(t))$ is bounded in $L^2_t(L^\infty_x)$. Hence we obtain from Gr\"owall's Lemma that
\[
\|v(t)\|_\infty \leq C + M(t) \|\partial_x u(t,x)\|_\infty,
\]
where $M$ is a bounded function of $L^2((0,\infty),\R)$. Note that $C$ here depends on $\|v(0)\|_\infty$ which depends on $\|c\partial_x u_0A_0[u_0]\|_\infty$, which can be bounded by $\|u_0\|_{C^{1,\alpha}}$. The result then follows thanks to Proposition \ref{prop:regsmooth}.
%
%

\end{proof}
We then obtain easily the
\begin{Prop}\label{prop:aubinlions}
Under the assumptions of the previous result, there exists $C>0$ which depends only on $T,c,B$, $\| u(0,\cdot)\|_{C^{1,\alpha}}$ and $\eta \in (0,1)$ such that the norm of $u$ in $C^\eta([0,T],C^\eta(\R))$ is bounded by $C$.
\end{Prop}
\begin{proof}
Recalling Propositions \ref{prop:regsmooth} and \ref{prop:estlip} there exists $C>0$ such that
\[
\|\partial_t u \|_{\infty} + \|\partial_x u\|_\infty \leq C.
\]
Hence, since $\|u\|_\infty \leq \|u|_{t = 0}\|_\infty$, we deduce from the so-called Aubin-Lions Lemma the required result.
\end{proof}

%
%
%
Putting together Propositions \ref{prop:estlip} and \ref{prop:aubinlions} and Theorem \ref{thm:gronwallvisc}, we then arrive at the following.
\begin{Theorem}
Assume that \eqref{Bdef} and \eqref{Bcomp} hold and that $u_0$ is a non-decreasing function such that $\|u_0\|_{C^{1,\alpha}} < \infty$. Then, there exists a unique viscosity solution $u$ of \eqref{eq143} which is continuous and such that $\partial_xu$ is a bounded function, locally in time.
\end{Theorem}
The proof of this result is done by following the lines of Theorem \ref{thm:reg}, namely by an approximation of the equation with smooth data and the addition of the term in $-\eps \partial_{xx}$ to obtain the existence. We do not provide it as it follows the same argument as before.

\subsection{Comments on the cases in which $c(\cdot)$ vanishes}
Let us note that the case in which $c(\cdot)$ vanishes can lead to more complex situations. To justify our claim, we recall the so-called Wishart case. In this situation, the PDE at interest is given by
\be\label{wishart}
\partial_t u + x\partial_x u\int_{\R_+}\frac{u(x)-u(y)}{(x-y)^2}dy + (\eta-1)\partial_xu + b(t,x)\partial_x u= 0 \text{ in } (0,\infty)\times (0,\infty).
\ee
In the previous, $\eta \geq 1$ is a parameter of the model, $b$ is a standard drift term and the previous equation is associated to the Dirichlet boundary condition 
\[
u(t,0) = 0 \text{ for all } t \geq 0.
\]
When $b$ is simply given by  $b(t,x) = -x$, there are explicit stationary solutions of the previous equation. There are given by means of the so-called Marcenko-Pastur distributions, which we recall here. Consider the density $m : (0,\infty) \to \R_+$ given by
\[
m(x) = \frac{\eta\sqrt{(\lambda_+ - x)(x-\lambda_-)}}{2\pi x}\mathbb{1}_{[\lambda_-,\lambda_+]}(x),
\]
where $\lambda_\pm = (1 \pm \sqrt{\eta^{-1}})^2$. Then, define $F(x) = \int_0^x m(y)dy$. This function $F$ satisfies for all $x > 0$
\[
\partial_x F(x)\left(  \eta-1 -x + x \int_{\R_+}\frac{F(x) - F(y )}{(x-y)^2}dy \right) = 0.
\]
Hence, $F$ is a stationary solution of \eqref{wishart} when $b(t,x) = -x$.

Note that for any $\eta > 1$, this stationary solution is Lipschitz continuous, but that this property fails for $\eta =1$, because $F$ is only $C^\frac12$. A singularity is then present at $x = 0$, precisely the point where both $c(\cdot)$ and $b(\cdot)$ vanish. In our opinion, this singularity is a strong argument in favor of the non-propagation of Lipschitz regularity. Even though, we insist upon the fact that even if: i) $F$ is the unique stationary solution of this equation, ii) it is the limit in long time of any solution of \eqref{wishart}, it is still not sufficient to establish the fact that we cannot propagate Lipschitz regularity as it could simply deteriorate in time.

In this case, because the singularity happens at the boundary of the domain of interest, it is fairly immediate to overcome this possible singularity at the boundary by using the Dirichlet boundary conditions. However, in general, if $c(\cdot)$ vanishes inside of the domain, the situation seems to require new arguments.

\subsection{The case of a singular drift}
We explain in this Section how the previous study naturally extends to the case of 
\be\label{dysonb}
\partial_t u + c(x)\partial_x u A_0[u] + b(x) \partial_x u = 0 \text{ in } (0,\infty)\times \R,
\ee
where $c$ is still a smooth function satisfying $\inf c > 0$ but now $b : \R \to \R$ is only a measurable function on which we impose that there exists $C_b > 0$ such that
\be
\ba
&|b(x)| \leq C_b, \text{ a. e. for } x \in \R,\\
&b + C_b Id \text{ is a non-decreasing function}.
\ea
\ee
The last assumption on $b$ implies in particular that for any $x\in \R$, one has $b(x^-):=\lim_{y \to x^-, y \ne x}b(y) \leq \lim_{y \to x^+, y \ne x}b(y)=:b(x^+)$. With a slight abuse of notation, we shall denote the interval $[b(x^-),b(x^+)] = b(x)$. We can then reformulate the second assumption on $b$ with
\be\label{bcompar}
\forall x,y \in \R, p \in b(x), q \in b(y) \quad (p-q)(x-y) \geq -C(x-y)^2,
\ee
\be\label{best}
\partial_x b \geq -C.
\ee
\begin{Rem}
More general interaction kernel could have been considered following exactly Section \ref{sec:reg}.
\end{Rem}
We refer the interested reader to \citep{lions2021cours} for a systematic study of transport equation with drifts presenting this kind of regularity. In our framework, we adapt slightly the notion of viscosity solution we already gave for this case. Namely, we are concerned with the following Definition.
\begin{Def}
An usc function $u : \R_+\times \R$ is a viscosity sub-solution of \eqref{dyson:gen} if for any $T > 0$, for any $\phi \in C^{1,1}_b$, $0 < \delta \leq \infty$  $(t_0,x_0) \in (0,T]\times \R$ point of maximum (resp. minimum) of $u-\phi$, there exists $p_0 \in b(x_0)$ such that
\[
\ba
\partial_t& \phi(t_0,x_0)  + \partial_x\phi(t_0,x_0)c(x_0) \left( A_{-\delta}[\phi(t_0)](x_0) + A_{\delta}[u(t_0)](x_0)\right) \\
&+ p_0\partial_x \phi(t_0,x_0)\leq 0 \text{ (resp. } \geq 0).
\ea
\]
\end{Def}
We can provide the following result.
\begin{Theorem}\label{thm:extensionb}
Given a non-decreasing bounded initial condition $u_0$, there exists a unique bounded, non-decreasing in $x$, viscosity solution of \eqref{dysonb}.
\end{Theorem}

The uniqueness part relies essentially on the stability of viscosity solutions as well as on the following Lemma.

\begin{Lemma}\label{lemma:complip}
Let $u_1$ and $u_2$ be two bounded functions, non-decreasing in $x$ for all $t \geq 0$, which are respectively viscosity sub and super solutions of \eqref{dysonb}. Assume that one of them is, uniformly in time, Lipschitz continuous in $x$. If for all $x \in \R$, $u_1(0,x) \leq u_2(0,x)$, then for all $t \geq 0, x \in \R$, $u_1(t,x) \leq u_2(t,x)$.
\end{Lemma}
\begin{proof}
The proof of this Lemma is a straightforward extension of Lemma 4.1 in \citep{bertucci2022spectral}. Indeed, the localization argument can be done exactly as in the proof of Theorem \ref{thm:gronwallvisc}, namely by the addition of the term in $r\psi$. The terms involving $b$ will simply be $\epsilon^{-1}(b(x_*)- b(y_*))(x_* -y_*)$, where $b(x_*)$ and $b(y_*)$ could be replaced by any element in those sets. By the assumptions on $b$, this term is greater that $-C_b \epsilon^{-1}(x_*-y_*)^2$ which converges toward $0$ as $\epsilon \to 0$. Hence the proof can be carried on in the same manner.
\end{proof}
In this case, the stability of viscosity solutions is expressed in the next result.
\begin{Lemma}\label{lemma:stab}
Let $(u_n)_{n \geq 0}$ be a bounded sequence of bounded viscosity sub solutions of \eqref{dysonb} (resp. viscosity super solutions). The function $u^* := \limsup_{n \to \infty} u_n$  (resp. $u^* = \liminf_{n \to \infty}u_n$ is a viscosity sub solution (resp. a viscosity super solution) of \eqref{dysonb}.
\end{Lemma}
\begin{proof}
Let $\phi \in \mathcal{C}^{1,1}_b$ such that $u^*- \phi$ has a point of maximum at $(t_*,x_*)$. Without loss of generality, we can assume that this is a point of strict maximum. Once again, without loss of generality, we can assume that $u_n - \phi$ has a point of maximum at some $(t_n,x_n)$. Remark that $(t_n,x_n) \to (t_*,x_*)$ as $n \to \infty$. Since $u_n$ is a viscosity sub solution of \eqref{dysonb}, it holds that for all $n \geq 0$, there exists $p_n \in b(x_n)$ such that 
\[
\partial_t \phi(t_n,x_n) + c(x_n)\partial_x \phi(t_n,x_n) A_0[\phi(t_n,\cdot)](x_n) + b(x_n)\partial_x \phi(t_n,x_n) \leq 0.
\]
Thanks to the regularity of $b$, we deduce that, extracting a subsequence if necessary, $p_n \to p_* \in b(x_*)$ as $n \to \infty$. Passing to the limit in the equation, we deduce that $u^*$ is viscosity sub solution of \eqref{dysonb}. The case of viscosity super solutions follows the same argument.
\end{proof}
With the help of those two results, the uniqueness part of Theorem \ref{thm:extensionb} is proven exactly by following the lines of Theorem 6 in \citep{bertucci2022spectral}.
Concerning the existence part, namely for general initial conditions, the proof now follows.
\begin{proof}
Consider $u_{0,\eps}$ and $u_0^\eps$ two Lipschitz non-decreasing functions such that for all $\eps > \eps' >0$, $u_{0,\eps} \leq u_{0,\eps'} \leq u_0 \leq u_0^{\eps'} \leq u_0^\eps$ together with $u_0^\eps - u_{0,\eps} \leq \eps$. For any $\epsilon > 0$, denote by $u_\eps$ (resp. $u^\eps$) the solution of \eqref{dysonb} with initial condition $u_{0,\eps}$ (resp. $u^\eps$). From Lemma \ref{lemma:complip}, we know that $(u_{\eps})_{\eps > 0}$ (resp. $(u^\eps)_{\eps > 0}$) is a non-decreasing (resp. non increasing) sequence. Hence it converges toward some function $u_*$ (resp. $u^*$). We clearly have $u_* \leq u^*$. On the other hand, thanks to Lemma \ref{lemma:stab}, we know that the lsc regularization of $u_*$ is a viscosity super solution of \eqref{dysonb} and that the usc regularization of $u^*$ is a viscosity sub solution of \eqref{dysonb}, hence we deduce from the comparison principle that $u_* \geq u^\eps - \epsilon$ since the inequality holds at the initial time. Passing to the limit $\eps \to 0$, we deduce that $u_* = u^*$ is a viscosity solution of \eqref{dysonb} with initial condition $u_0$.
\end{proof}

\section{Some identities for the Dyson flow}\label{sec:identities}
In this Section, we prove various results around the so-called Dyson flow, i.e. the equation
\be\label{dysoneq}
\partial_t m + \partial_x(mH[m]) = 0 \text{ in } (0,\infty)\times \R,
\ee
when the initial condition $m_0$ is in $\mathcal{P}(\R)$.
In each of the following cases, the strategy of proof is the same: we show a certain property for smooth solutions of \eqref{dysoneq} ; we then proceed by a regularizing argument to show that the associated property is also valid for solutions $m$ characterized by the fact that $u(t,x) = m(t,(-\infty,x])$ is a viscosity solution of \eqref{dysonint}. Since we detailed such an approximation in Section \ref{sec:reg}, we do not detail it here.

\subsection{Decay of $L^p$ norms}
We prove here the following statement.
\begin{Prop}
For any $m_0 \in \mathcal{P}_2(\R)$, the unique solution of \eqref{dysoneq} with initial condition $m_0$ is such that for any $0 \leq t \leq s$, $1 \leq p \leq \infty$
\[
\|m(t)\|_p \geq \|m(s)\|_p.
\]
\end{Prop}
\begin{proof}
We only prove the statement for smooth solutions, the general results can be obtained by approximation like the previous one, provided some sort of continuity in time for the narrow convergence of probability measures. Such a continuity holds since $m_0 \in \mathcal{P}_2(\R)$. The relation is clearly true for $p=1$ and, observing the proof of Proposition \ref{prop:regsmooth} in the case of \eqref{dysoneq}, it is also true for $p = \infty$. Hence, it only remains to prove it for $1 < p < \infty$. Multiplying \eqref{dysoneq} by $m^{p-1}$ and integrating we obtain
\[
\ba
\frac{d}{dt} \|m(t)\|_p^p &= \int_\R (\partial_x) m m^{p-1} H[m]\\
&= \frac 1p\int_\R\partial_x(m^p)H[m]\\
&= -\frac 1p\int_\R m^pA_0[m]\\
& = -\frac 1p\int_\R m^p(x) \int_\R\frac{m(x) - m(y)}{(x-y)^2}dydx\\
& = -\frac{1}{2p}\int_\R\int_\R \frac{(m(x)-m(y))(m^p(x)-m^p(y))}{(x-y)^2}dxdy.
\ea
\]
The last term being clearly non-positive since $m \geq 0$, the result follows.
\end{proof}

\subsection{Around the free entropy}

We introduce the function $E : \mathcal{P}_2(\R) \to \R\cup \{-\infty\}$ which is often called the free entropy of the physical system associated to the Dyson flow. It is defined by
\[
E(µ) = \frac 12 \int_\R \int_\R \log(|x-y|)µ(dx)µ(dy).
\]

For sufficiently smooth $µ$, we have the following
\[
\nabla_µ E(µ,x) = \int_\R\log(|x-y|)µ(dy),
\]
\[
D_µ E(µ,x) = H[µ](x).
\]
Hence, this quantity is deeply linked with \eqref{dysoneq}. We shall prove later on that it is increasing and continuous along the Dyson flow. We begin by remarking that $E$ can be easily expressed in terms of the Fourier transform. Indeed, denote by $\hat f$ the normalized Fourier transform of a distribution $f$ and by $g(x) = \log(|x|)$. Recall the classical relation
\[
\hat{g }(\xi) = -\frac 12 f.p. \left(\frac{1}{|\xi|}\right) -\gamma \delta_0(\xi),
\]
where $f.p. \left(\frac{1}{|\xi|}\right)$ denotes the finite part of $\xi \to |\xi|^{-1}$ and $\gamma$ is the Euler constant. Hence, since $E(m) = \frac 12 \langle m*g,m\rangle_{L^2}$, it holds that
\[
E(m) = -\frac 14 \left\langle f.p. \left(\frac{1}{|\xi|}\right), \hat{m}^2\right\rangle_{L^2} - \gamma.
\]
This remark leads us to the following result.
\begin{Prop}\label{prop:E}
\begin{itemize}
\item Take $m \in \Pt$, then $E(m) > -\infty$ if and only if $m \in H^{-\frac 12}$.
\item For any $m \in \Pt$ such that $E(m) > -\infty$, there exists a sequence $(m^n)_{n \geq 0}$, valued in $\Pt$, with smooth bounded densities, such that $m^n  \rightharpoonup m$ and $E(m^n) \to E(m)$ as $n \to \infty$.
\item The function $E(\cdot)$ is upper semi-continuous for the narrow topology.
\end{itemize}
\end{Prop}
\begin{proof}
Since $m \in \Pt$, it follows that $\hat{m}$ belongs to $C^2_b(\R)$. Hence 
\[
 \left\langle f.p. \left(\frac{1}{|\xi|}\right), \hat{m}^2\right\rangle_{L^2} < \infty  \iff \|m\|_{H^{-\frac 12}}^2:=  \int_\R \frac{\hat{m}(\xi)^2}{\sqrt{1 + |\xi|^2}} < \infty.
\] 
This proves the first part of the claim. For the second part, it suffices to remark that the result is true if we can choose $(m^n)_{n \geq 0}$ such that 
\[
\frac 14 \left\langle f.p. \left(\frac{1}{|\xi|}\right), \hat{m^n}^2\right\rangle_{L^2} \quad\underset{n \to \infty}{\longrightarrow} \frac 14 \left\langle f.p. \left(\frac{1}{|\xi|}\right), \hat{m}^2\right\rangle_{L^2},
\]
which is true since $m\in \Pt$ and $E(m) > - \infty$. The third part of the claim is obtained by remarking that $(x,y) \to \log(|x-y|)$ is upper semi continuous.
\end{proof}
The first part of this result states that $E(\cdot)$ can indeed be used to measure some regularity while the second one shall be helpful later on. We now pass to the main result of this Section.
\begin{Prop}
Let $(m_t)_{t \geq 0}$ be a smooth solution of \eqref{dysoneq} valued in $\Pt$ such that $E(m_0) > - \infty$. Then for any $0 \leq t' \leq t$
\[
E(m_t) = E(m_{t'}) + \int_{t'}^t\int_\R H[m_{s}]^2m_s ds.
\]
\end{Prop}
\begin{proof}
We start by proving the claim in the case of smooth solutions. We can compute
\be
\ba
\frac{d}{dt} E(m_t) &= \int_R \nabla_µE(m_t,x)\partial_tm(t,x)dx\\
& = \int_\R D_µE(m_t,x)H[m_t](x)m(t,x)dx.
\ea
\ee
From which we deduce
\[
 \frac 12 \int_\R \int_\R \log(|x-y|)m_t(dx)m_t(dy) =  \frac 12 \int_\R \int_\R \log(|x-y|)m_{t'}(dx)m_{t'}(dy) + \int_{t'}^t\int_\R H[m_{s}]^2m_s ds.
\]
It remains to justify the relation in the case of non smooth solutions. Let $m_0 \in  \Pt$ and assume first that $m_0 \in L^\infty$. Take a sequence $(\rho_\eps)_{\eps > 0}$ of smooth functions which converges in $L^p_{loc}$ toward $\log$, for any $p \in [1,\infty)$, which is of the form $\rho_\eps = \log * \tilde{\rho}_\eps$. It then follows that
\[
\ba
 \frac 12 \int_\R \int_\R \rho_\eps(|x-y|)m_t(dx)m_t(dy) =  \frac 12 \int_\R \int_\R \rho_\eps(|x-y|)m_{t'}(dx)m_{t'}(dy) + \int_{t'}^t\int_\R H[\tilde{\rho}_\eps*m_s]H[m_{s}]m_s ds.
 \ea
\]
Hence, the claim is proved for $m_0 \in L^\infty\cap\mathcal{P}_2$ if we can pass to the limit $\epsilon \to 0$ in the previous inequality. The passage to the limit is indeed valid in the first two terms from $m_t,m_{t'} \in L^\infty$ and the fact that $m_t,m_{t'} \in \mathcal{P}_2(\R)$. For the third terms, it suffices to recall that $H$ is continuous in any $L^p$, $1 < p < \infty$ and that $\tilde{\rho_\eps}*m_s\to m_s$ as $\eps \to 0$, uniformly in $s$, in $L^p$ for any $1< p < \infty$.\\

Consider now a general $m_0 \in \mathcal{P}_2(\R)$ with $E(m_0) > -\infty$ and a sequence $(m^n_0)$ as in Proposition \ref{prop:E}. For all $n \geq 0$, $t \geq 0$, we have
\[
E(m^n_t) = E(m^n_0) +\int_{0}^t\int_\R H[m^n_{s}]^2m^n_s ds.
\]
Passing to the limit $n \to \infty$, we deduce that for all $t \geq 0$
\[
E(m_t) -E(m_0) \geq \int_{0}^t\int_\R H[m_{s}]^2m_s ds.
\]
Hence, $(E(m_t))_{t \geq 0}$ is indeed increasing in time. Since it is upper semi continuous for the weak topology, we deduce that it is continuous at $t = 0$. On the other hand, thanks again to the $L^\infty$ regularizing effect, we have for any $0 < t' \leq t$.
\[
E(m_t) = E(m_{t'}) + \int_{t'}^t\int_\R H[m_{s}]^2m_s ds.
\]
Thus we finally obtain the result by passing to the limit $t' \to 0$.
\end{proof}

\begin{Rem}
In this case, the approximation is made on the $\log$ and not on the solution of the equation itself.
\end{Rem}
\begin{Cor}
Let $m_0 \in \Pt$ be such that $E(m_0) > -\infty$. Then $(E(m_t))_{t \geq 0}$ is a continuous increasing function.
\end{Cor}
\begin{Rem}
Let us insist upon the fact that the continuity was not obtained for the $L^p$ norm in the previous subsection.
\end{Rem}

\subsection{Contraction in the Wasserstein space}
We also recall the following result, of which we provide a new proof in the case $p=2$.
\begin{Prop}\label{prop:decreas}
Let $1 \leq p \leq \infty$, and $(µ_t)_{t \geq 0}$, $(\nu_t)_{t \geq 0}$ be two solutions of \eqref{dyson} valued in $\mathcal{P}_p$, then for any $t \leq s$, $W_p(µ_s,\nu_s) \leq W_p(\mu_t,\nu_t)$.
\end{Prop}
We start by recalling Cotlar's identity for smooth functions $u : \R \to \R$,
\[
H[u]^2 = \pi^2u^2 + 2H[uH[u]].
\]
Multiplying by $u$ and integrating, we obtain
\[
\ba
\int_\R H[u]^2 u &= \pi^2\int_\R u^3 + 2 \int_\R H[u]^2 u\\
& = \frac{\pi^2}{3} \int_\R u^3.
\ea
\]

In particular, for a probability measure $m \in \mathcal{P}$, $H[m] \in L^2(m)$ if and only if $m \in L^3$.

\begin{proof}
We start with the case in which $(µ_t)_{t \geq 0}$ and $(\nu_t)_{t \geq 0}$ are valued in $L^3$. Consider an optimal coupling $\gamma$ between $µ_0$ and $\nu_0$ for $2$-Wasserstein distance and $(X_0,Y_0)$ a couple of real random variables on a standard probability space such that $\mathcal{L}((X_0,Y_0)) = \gamma$. Because the processes are valued in $L^3$, we deduce from Cotlar's identity that $H[\mathcal{L}(X_0)](X_0)$ is a squared integrable random variable. Consider now the ODEs
\[
\frac{d}{dt} X_t = H[\mathcal{L}(X_t)](X_t) \quad ; \quad \frac{d}{dt} Y_t = H[\mathcal{L}(Y_t)](Y_t).
\]
By construction, for any $t \geq 0$, $\mathcal{L}(X_t) = µ_t$ and $\mathcal{L}(Y_t) = \nu_t$. It then follows that for all $t \geq 0$
\[
\ba
\frac{d}{dt}W_2^2(µ_t,\nu_t)|_{t = 0} \leq \frac{d}{dt}\mathbb{E}[|X_t - Y_t|^2]|_{t = 0}.
\ea
\]
We now compute
\[
\ba
\frac12\frac{d}{dt}\mathbb{E}[|X_t - Y_t|^2]|_{t = 0} &= \mathbb{E}[(X_0-Y_0)(H[\mathcal{L}(X_0)](X_0) - H[\mathcal{L}(Y_0)](Y_0)]\\
& = \int_\R\int_\R\int_\R\int_\R (x-y)\left(\frac{1}{x-x'} - \frac{1}{y -y'}\right)\gamma(dx,dy)\gamma(dx',dy')\\
& = \int_\R\int_\R\int_\R\int_\R \frac{-(x-y)^2 + (x-y)(x'-y')}{(x-x')(y-y')} \gamma(dx,dy)\gamma(dx',dy')
\ea
\]
Recall that since $\gamma$ is an optimal coupling, $(x-x')(y-y') \geq 0$ holds $\gamma\times \gamma$ almost everywhere. Using $2(x-x')(y-y') \leq (x-y)^2 + (x'-y')^2$, we obtain that
\[
\ba
\frac 12\frac{d}{dt}\mathbb{E}[|X_t - Y_t|^2]|_{t = 0} &\leq \frac 12\int_\R\int_\R\int_\R\int_\R \frac{(x'-y')^2-(x-y)^2 }{(x-x')(y-y')} \gamma(dx,dy)\gamma(dx',dy').
\ea
\]
The term on the right side vanishes by symmetry, hence $\frac{d}{dt}\mathbb{E}[|X_t - Y_t|^2]|_{t = 0} \leq 0$ from which the result follows since the same argument can be made for any time.\\

In the general case, thanks to the $L^\infty$ regularizing property, we deduce that for any $\epsilon > 0$ such that $t+ \epsilon \leq s$, we have $W_2(µ_s,\nu_s) \leq W_2(µ_{t+\eps},\nu_{t + \eps})$. Hence, by continuity of the Dyson flow in $(\Pt,W_2)$, we obtain the required result.
\end{proof}

\begin{Rem}
A similar approach could have been used in the case $p \ne 2$ but we do not present it here.
\end{Rem}

\section{Complements on modeling}\label{sec:model}
We now present various independent extensions of the Dyson model.
\subsection{A bulk of eigenvalues interacting with spikes}
If the $N\times N$ matrix from which derives the spectral measure $m$ is perturbed by a matrix of rank one, it has been observed in \citep{baik}, that we can identify an outlier of the spectrum. That is an eigenvalue which corresponds to the perturbation. We explain the effect of a similar perturbation in this dynamical setting.
Consider that the usual $N\times N$ Dyson model is perturbed by the addition of the matrix $a_t e_1\otimes e_1$, for $(a_t)_{t \geq 0}$ is a given smooth process. Since the rank of this perturbation is only one, it does not alter the limit spectral measure, which is then given by the usual Dyson equation. However, if we can identify the "outlier" associated to this perturbation, its evolution shall naturally depend on the rest of the spectral measure. This leads us to the system
\be\label{spike}
\ba
d\lambda_t& = H[m_t](\lambda_t) dt + da_t,\\
\partial_t m& + \partial_x(mH[m])  = 0 \text{ in } (0,T)\times \R,\\
\lambda_0& = a_0 +H[m_0](\lambda_0), \quad m|_{t = 0} = m_0.
\ea
\ee
Note that the main mathematical difficulty in this system lies in the fact that we do not know in general if $H[m_t]$ is sufficiently smooth to characterize $(\lambda_t)_{t \geq 0}$ as the unique solution of the first equation. Indeed, even using the regularity proved in \citep{biane1997free}, we would only have a H\"older estimate on $H[m_t]$ which would give existence but not necessary uniqueness of such a solution.\\

A case which is of particular importance in our opinion is the one in which the special eigenvalue $\lambda$ is greater than all the other one, more specifically when $m_t$ is supported on a set of the form $(-\infty,R_t]$ for some $R_t \in \R$ and when $\lambda_t > R_t$. Indeed, in such a context, $H[m_t]$ is locally Lipschitz and decreasing around $\lambda_t$ as the next elementary result states.
\begin{Lemma}
Let $m \in \mathcal{P}(\R)$ be supported on $(-\infty,R_0]$. Then, $H[m]$ is decreasing on $[R_0,\infty)$ and it is $C$ Lipschitz continuous on $[R_0+\alpha,\infty)$ and $C$ depends only $\alpha > 0$.
\end{Lemma}
\begin{proof}
Let $m$ be such a measure. Take $\lambda\geq \lambda' \geq R_0$ and compute
\[
\ba
H[m](\lambda) - H[m](\lambda') = (\lambda' - \lambda)\int_\R\frac{1}{(\lambda-x)(\lambda'-x)}m(dx).
\ea
\]
Hence the result follows since for all $x$ in the support of $m$, $x \leq R_0 \leq \lambda'\leq \lambda$.
\end{proof}

This implies that if we can guarantee that $\lambda$ stays strictly outside the support of $m$, then the ODE satisfied by $\lambda$ is well-posed under standard assumptions on $(a_t)_{t \geq 0}$. Note that in several situations, this can be checked by an explicit computation. Indeed, consider for instance the case in which $m_0$ is supported on $(-\infty,R_0]$. Then, thanks to Proposition 4.4 in \citep{bertucci2022spectral}, we can compare the cumulative distribution functions of the solutions of \eqref{dysoneq} associated to $m_0$ and $\delta_{R_0}$. The latter is given by the semi-circular law. Hence we know that for all $t \geq 0$, $m_t$ is supported on $(-\infty,R_0 + \sqrt{t}]$. Hence, if $\lambda_0 > R_0$ and $da_t \geq \frac{1}{2\sqrt{t}}$, we know that for all $t > 0$, $\lambda_t > R_t$.

An interesting feature is that in general, the fact that $(\lambda_t)_{t \geq 0}$ stays outside of the support of $m_t$ does not simply follow from the fact that this holds at the initial time as the next result shows.
\begin{Prop}
Consider $m_0 = \delta_0$ and $\lambda_0 > 0$ and assume that $(a_t)_{t \geq 0}$ is constant with $a_0 > 0$. The system \eqref{spike} is locally well posed in time and there exists $t_0 > 0$ such that $\lambda_{t_0} = \sqrt{t_0}$.
\end{Prop}
\begin{proof}
Since $\lambda_0 > 0$ the system is well posed at least locally in time. Introduce $Z_t := \lambda_t^2-t$. We want to show that there exists $t_0$ such that $Z_{t_0} = 0$. Since $m_0$ is Dirac mass, the solution of the Dyson equation is given by $m_t(dx) = \frac{2}{\pi t}\sqrt{t - x^2}\mathbb{1}_{[-\sqrt{t},\sqrt{t}]}(x)dx$ and thus for any $x > \sqrt{t}$
\[
H[m_t](x) = \frac{x - \sqrt{x^2 -t}}{2t}.
\]
Hence we deduce that 
\[
\ba
\frac{d}{dt} Z_t &= \lambda_t \frac{\lambda_t - \sqrt{\lambda_t^2 -t}}{t} -1\\
&= \frac{ Z_t - \lambda_t \sqrt{Z_t}}{t}\\
&= \frac{Z_t - \sqrt{Z_t + t}\sqrt{Z_t}}{t}
\ea
\]
Hence $(Z_t)$ is decreasing, while it is positive. Remark that it cannot have a positive limit. Moreover, as it gets closer to $0$, the previous ODE as a similar behaviour as $\dot{X} \leq \sqrt{X}\sqrt{t}^{-1}$, which reaches $0$ in finite time, hence the result.
\end{proof}

Note that a quite similar question, which is addressed in more details, in the case of large but finite matrix, was studied in \citep{dubach}. However, the authors considered that the "bulk" is constant in time in their model.

\subsection{An extension to other non-linearities}
Consider the equation 
\be
\partial_t m + \partial_x (m\sigma(m)H[m]) = 0 \text{ in } (0,\infty)\times \R,
\ee
where $\sigma : \R_+ \to \R_+$ is a continuous function. Then we claim than an analysis of this equation by means of the viscosity solutions of the PDE
\be\label{eq:dysonFsigma}
\partial_t u + \sigma(\partial_x u)\partial_xuH[\partial_x u] = 0 \text{ in } (0,\infty)\times \R
\ee
can be done following the lines of \citep{bertucci2022spectral}. Indeed, the function $\sigma$ does not perturb the proof of the comparison of a viscosity super-solution and a viscosity sub-solution. Note that this is due in large part to the fact that, since we are considering here the exact Hilbert transform, and not some other singular operators, we do not need to show that the Lipschitz regularity propagates in \eqref{eq:dysonFsigma}. We can provide the following result.
\begin{Prop}
Let $u_1$ be a viscosity super-solution of \eqref{eq:dysonFsigma} and $u_2$ a viscosity sub-solution of the same equation. Assume moreover that both are non-decreasing in $x$. Then if for all $x \in \R$, $u_1(0,x) \geq u_2(0,x)$, it holds that for all $t \geq 0,x \in \R$, $u_1(t,x) \geq u_2(t,x)$. In particular, there exists at most one viscosity solution of \eqref{eq:dysonFsigma} given an initial condition.
\end{Prop}
\begin{proof}
We only sketch the main lines of the proof. For $\epsilon > 0$, consider a point $(t_*,s_*,x_*,y_*)$ of maximum of 
\[
(t,s,x,y) \to u_1(t,x) - u_2(s,y) - \frac{1}{2\eps}((t-s)^2 + (x-y)^2),
\]
and assume that $t_*,s_* > 0$. Using that $u_1$ is a viscosity sub-solution and $u_2$ a viscosity super-solution, we obtain that for all $\delta > 0$
\[
\eps^{-1}(t_*-s_*)  + \sigma(\eps^{-1}(x_*-y_*))\eps^{-1}(x_*-y_*)\left(A_{-\delta}[\phi_\eps](x_*) + A_\delta[u_1(t_*)](x_*)\right) \leq 0,
\]
\[
\eps^{-1}(t_*-s_*)  + \sigma(\eps^{-1}(x_*-y_*))\eps^{-1}(x_*-y_*)\left(A_{-\delta}[\psi_\eps](y_*) + A_\delta[u_2(s_*)](y_*)\right) \geq 0,
\]
where $\phi_\eps(x) = \frac{1}{2\eps}(x-y_*)^2$ and $\psi_\eps(y) = -\frac{1}{2\eps}(y-x_*)^2$. Since $u_1(t_*,\cdot)$ is non-decreasing we deduce as usual that $x_* \geq y_*$. Recall that from the definition of the point of maximum, $A_\delta[u_1(t_*)](x_*) - A_\delta[u_2(s_*)](y_*) \leq 0$. Hence, taking the difference of the two relations leads to 
\[
C\sigma(\eps^{-1}(x_*-y_*))\eps^{-1}(x_*-y_*)\delta \leq 0,
\]
where we used the regularity of $\phi_\eps$ and $\psi_\eps$ to estimate the terms in $A_{-\delta}$. The previous is the main argument to prove a comparison principle and the rest of the argument follows what we did in the proof of Theorem \ref{thm:gronwallvisc}, namely to localize the point of maximum or to arrive at a contradiction if it not reached for $t_*s_* = 0$.
\end{proof}

\subsection{Two systems in interaction}
Another natural extension concerns the interaction of two systems governed by two equations of the type of \eqref{dyson}. Namely, consider the system
\be\label{dyson:system}
\partial_t m_i + \partial_x(m_iH[m_i]) + \partial_x(b_i[m_1,m_2]m_i) = 0 \text{ in } (0,\infty)\times \R,
\ee
for $i \in \{1;2\}$. This system models two systems who are each driven by the Dyson Brownian motion but who also feels an additional force ($b_i$ for the equation $i$) which depends on the state of the two systems. The following holds.
\begin{Prop}
Assume that $b_1,b_2 : \mathcal{P}(\R)^2 \to C^{0,1}(\R)$ are smooth and bounded operators, and consider two probability measures $m_{0,1}$ and $m_{0,2}$ in $\mathcal{P}_2(\R)$. Then, there exists a unique solution $(m_1,m_2)$ of \eqref{dyson:system} which is characterized as being the derivative of the couple $(F_1,F_2)$, unique viscosity solutions of 
\[
\partial_t F_i + \partial_x F_iH[\partial_x F_i] + b_i(\partial_x F_1,\partial_x F_2)\partial_x F_i  = 0 \text{ in } (0,\infty)\times \R,
\]
with initial condition $F_i(0,x) = m_{0,i}((-\infty,x])$.
\end{Prop}
\begin{proof}
The existence follows easily from a compactness argument. Fix $T > 0$ and denote by $B := b_1(\mathcal{P}(\R))\cup b_2(\mathcal{P}(\R))$. Take $b : \R_+\to B$ and consider the solution $m^b$ of 
\[
\partial_t m^b + \partial_x(m^bH[m^b]) + \partial_x(bm^b) = 0 \text{ in } (0,T)\times \R.
\]
Take a time $t \in [0,T]$ and an optimal coupling $(X_t,Y_t)$ for the squared Wasserstein distance between $m^{b'}(t)$ and $m^b(t)$. Consider now the SDE
\[
\begin{aligned}
&dX_s = H[\mathcal{L}(X_s)](X_s)ds + b'(s,X_s)ds\\
&dY_s = H[\mathcal{L}(Y_s)](Y_s)ds + b(s,Y_s)ds.
\end{aligned}
\]
Then it follows from the computation of Proposition \ref{prop:decreas} that
\[
\ba
\frac{d}{ds}\mathbb{E}[|X_s -Y_s|^2] &\leq \mathbb{E}[(X_s -Y_s)(b'(s,X_s) - b(s,Y_s)]\\
& =  \mathbb{E}[(X_s-Y_s)(b(s,X_s) - b(s,Y_s))] +  \mathbb{E}[(X_s-Y_s)(b'(s,X_s) - b(s,X_s))]  \\
&\leq C \mathbb{E}[|X_s-Y_s|^2]+ \frac{1}{2}\|b-b'\|_\infty^2.
\ea
\]
From Gr\"onwall's Lemma, we obtain that there exists $C$ such that 
\[
W_2^2(m^b(t),m^{b'}(t)) \leq (e^{Ct}-1)\|b-b'\|_\infty^2.
\]
Since there exists $C >0$ such that for all $b \in B$, 
\[
W_2^2(m^b(t), m^b(s)) \leq C\sqrt{|t-s|},
\]
we deduce the existence of a solution of the system from standard fixed point theorems.\\

The uniqueness part follows the line of the Gr\"onwall estimate used in the proof of Theorem \ref{thm:gronwallvisc}. Denote by $(m_1,m_2)$ and $(\tilde m_1,\tilde m_2)$ two solutions of the system \eqref{dyson:system}. We define $u_1(t,x) = \int_{-\infty}^xm(t,y)dy$ and similarly for $u_2,$ $\tilde u_1$ and $\tilde u_2$. Define $M(t) := \|(u_1(t),u_2(t))-(\tilde u_1(t),\tilde u_2(t))\|_\infty$. From the regularity we assumed on $b_1$ and $b_2$, we deduce (from the proof of Theorem \ref{thm:gronwallvisc}) that there exists $C > 0$ such that for $i=1,2$, $t\geq 0$
\[
\|u_i(t) - \tilde u_i(t)\|_\infty \leq C \int_0^tM(s)ds.
\]
Hence, we obtain the uniqueness of a solution of \eqref{dyson:system}
\end{proof}

\subsection{Reflexion at the boundary}
We explain how we can model the reflexion of eigenvalues at a certain boundary. This classical question in stochastic analysis is now well understood, especially since the seminal works \citep{lions1984stochastic,lions1981construction}. Namely, we want to explain why the reflexion of the eigenvalues of the matrix on a certain maximum level $R_0$ does not perturb the mathematical analysis. We do not provide the full study of such a phenomenon but rather explain why this kind of reflexion does not create a singularity near $R_0$. We follow the approach of \citep{lions1981construction}.

To make the matter more precise, consider an element $m_0\in \mathcal{P}(\R)$ with bounded density, still denoted $m_0$ and  consider for $\epsilon > 0$ the following equation
\be\label{reflexion:penal}
\partial_t m + \partial_x(mH[m]) + \partial_x\left( \frac 1 \eps (x - R_0)_+ m(t,x)\right) = 0 \text{ in } (0,\infty)\times \R.
\ee

This kind of penalization is well known to approximate the reflexion at $R_0$, or equivalently the Neumann Boundary condition at $R_0$. We can prove the following result.
\begin{Lemma}
For all $\eps > 0$, the solution $m$ of \eqref{reflexion:penal} satisfies for all $t\geq 0, x \geq 0$ $m(t,x) \leq \|m_0\|_\infty$.
\end{Lemma}
\begin{proof}
The result follows form a standard a priori estimate once again. It is immediate to check that, if $m$ is smooth, for any $t \geq 0$, at any point of local maximum $\bar x$ of $m(t,\cdot)$, 
\[
\partial_t m(t,\bar x) + m(t,\bar x) A_0[m(t,\cdot)](\bar x) + \frac 1 \eps \mathbb{1}_{\{ \bar x \geq R_0\}}m(t,\bar{x}) = 0.
\]
Which implies the result since the same kind of argument as in Lemma \ref{lemma:existsmooth} can be carried on to justify this computation.
\end{proof}
With the help of this Lemma, it is now entirely classical to establish the following.
\begin{Theorem}
Let $u_0 \in W^{1,\infty}(\R)$. The sequence $(u_\eps)_{\eps > 0}$ of solution of \eqref{reflexion:penal} converges locally uniformly toward the unique viscosity solution of 
\[
\ba
\partial_t u + \partial_x u A_0[u] = 0 \text{ in } (0,\infty)\times (-\infty,R_0),\\
\partial_x u (t,R_0) = 0 \text{ for all } t > 0.
\ea
\]
Moreover, this solution satisfies
\[
u(t,x) = 1 \text{ for all } t > 0, x \geq R_0.
\]
\end{Theorem}

\section*{Acknowledgments}
The authors have been partially supported by the Chair FDD/FIME (Institut Louis Bachelier). The first author has been partially supported by the Lagrange Mathematics and Computing Research Center.
\bibliographystyle{plainnat}
\bibliography{bibmatrix}

\begin{thebibliography}{13}
\providecommand{\natexlab}[1]{#1}
\providecommand{\url}[1]{\texttt{#1}}
\expandafter\ifx\csname urlstyle\endcsname\relax
  \providecommand{\doi}[1]{doi: #1}\else
  \providecommand{\doi}{doi: \begingroup \urlstyle{rm}\Url}\fi

\bibitem[Anderson et~al.(2010)Anderson, Guionnet, and
  Zeitouni]{anderson2010introduction}
Greg~W Anderson, Alice Guionnet, and Ofer Zeitouni.
\newblock \emph{An introduction to random matrices}.
\newblock Number 118. Cambridge university press, 2010.

\bibitem[Baik et~al.(2005)Baik, Ben~Arous, and P{\'e}ch{\'e}]{baik}
Jinho Baik, G{\'e}rard Ben~Arous, and Sandrine P{\'e}ch{\'e}.
\newblock Phase transition of the largest eigenvalue for nonnull complex sample
  covariance matrices.
\newblock \emph{Annals of Probability}, pages 1643--1697, 2005.

\bibitem[Bertucci et~al.(2022)Bertucci, Debbah, Lasry, and
  Lions]{bertucci2022spectral}
Charles Bertucci, M{\'e}rouane Debbah, Jean-Michel Lasry, and Pierre-Louis
  Lions.
\newblock A spectral dominance approach to large random matrices.
\newblock \emph{Journal de Math{\'e}matiques Pures et Appliqu{\'e}es},
  164:\penalty0 27--56, 2022.

\bibitem[Biane(1997)]{biane1997free}
Philippe Biane.
\newblock On the free convolution with a semi-circular distribution.
\newblock \emph{Indiana University Mathematics Journal}, pages 705--718, 1997.

\bibitem[Chan(1992)]{chan}
Terence Chan.
\newblock The wigner semi-circle law and eigenvalues of matrix-valued
  diffusions.
\newblock \emph{Probability theory and related fields}, 93\penalty0
  (2):\penalty0 249--272, 1992.

\bibitem[Dubach and Erd{\H{o}}s(2023)]{dubach}
Guillaume Dubach and L{\'a}szl{\'o} Erd{\H{o}}s.
\newblock Dynamics of a rank-one perturbation of a hermitian matrix.
\newblock \emph{Electronic Communications in Probability}, 28:\penalty0 1--13,
  2023.

\bibitem[Forcadel et~al.(2009)Forcadel, Imbert, and
  Monneau]{forcadel2009homogenization}
Nicolas Forcadel, Cyril Imbert, and R{\'e}gis Monneau.
\newblock Homogenization of some particle systems with two-body interactions
  and of the dislocation dynamics.
\newblock \emph{Discrete and continuous dynamical systems-series A},
  23\penalty0 (3):\penalty0 pp--785, 2009.

\bibitem[Imbert and Monneau(2008)]{imbert2008homogenization}
Cyril Imbert and R{\'e}gis Monneau.
\newblock Homogenization of first-order equations with-periodic hamiltonians.
  part i: Local equations.
\newblock \emph{Archive for Rational Mechanics and Analysis}, 187\penalty0
  (1):\penalty0 49--89, 2008.

\bibitem[Imbert et~al.(2008)Imbert, Monneau, and
  Rouy]{imbert2008homogenization2}
Cyril Imbert, R{\'e}gis Monneau, and Elisabeth Rouy.
\newblock Homogenization of first order equations with
  (u/$\varepsilon$)-periodic hamiltonians part ii: Application to dislocations
  dynamics.
\newblock \emph{Communications in Partial Differential Equations}, 33\penalty0
  (3):\penalty0 479--516, 2008.

\bibitem[Lions(2021-2022)]{lions2021cours}
Pierre-Louis Lions.
\newblock Cours au college de france.
\newblock \emph{Available at www. college-de-france. fr}, 2021-2022.

\bibitem[Lions and Sznitman(1984)]{lions1984stochastic}
Pierre-Louis Lions and Alain-Sol Sznitman.
\newblock Stochastic differential equations with reflecting boundary
  conditions.
\newblock \emph{Communications on pure and applied Mathematics}, 37\penalty0
  (4):\penalty0 511--537, 1984.

\bibitem[Lions et~al.(1981)Lions, Menaldi, and Sznitman]{lions1981construction}
Pierre-Louis Lions, Jose~Luis Menaldi, and Alain-Sol Sznitman.
\newblock Construction de processus de diffusion r{\'e}fl{\'e}chis par
  p{\'e}nalisation du domaine.
\newblock \emph{Comptes Rendus de l'Acad\'emie des Sciences de Paris},
  292:\penalty0 559--562, 1981.

\bibitem[Rogers and Shi(1993)]{rogers}
Leonard~CG Rogers and Zhan Shi.
\newblock Interacting brownian particles and the wigner law.
\newblock \emph{Probability theory and related fields}, 95\penalty0
  (4):\penalty0 555--570, 1993.

\end{thebibliography}

\end{document}